\documentclass[12pt]{amsart}

\usepackage{amssymb, amsthm, amsmath}

\newtheorem{theorem}{Theorem}[section]
\newtheorem{lemma}[theorem]{Lemma}

\theoremstyle{definition}
\newtheorem{definition}[theorem]{Definition}

\theoremstyle{remark}

\numberwithin{equation}{section}

\newcommand{\D}{\mathbb{D}}

\newcommand{\C}{\mathbb{C}}

\newcommand{\ip}[2]{\langle #1, #2 \rangle}

\title{A refined Agler decomposition and geometric applications}

\author{Greg Knese}

\address{University of Alabama, Tuscaloosa, AL, 35487-0350}

\date{\today}

\email{geknese@bama.ua.edu}


\thanks{This research was supported by NSF grant DMS-1001791}


\begin{document}
\bibliographystyle{apalike}
\maketitle

\begin{abstract} 
We prove a refined Agler decomposition for bounded analytic functions
on the bidisk and show how it can be used to reprove an interesting
result of Guo et al. related to extending holomorphic functions without
increasing their norm.  In addition, we give a new treatment of Heath
and Suffridge's characterization of holomorphic retracts on the
polydisk.
\end{abstract} 

\section{Introduction}
Let $\D$ denote the unit disk in $\C$ and $\D^2 = \D \times \D$ the
unit bidisk.

\cite{jA88} proved that a holomorphic function $f: \D^2 \to \D$ satisfies
a decomposition (later called an \emph{Agler decomposition}) of the
form
\[
1- f(z) \overline{f(\zeta)} = (1- z_1 \bar{\zeta}_1) K_1(z,\zeta) +
(1-z_2\bar{\zeta}_2) K_2(z,\zeta)
\]
where $K_1, K_2$ are positive semi-definite kernel functions.
A kernel function $K:\Omega \times \Omega \to \C$ is \emph{positive
  semi-definite} if for every finite subset $F \subset \Omega$ the
matrix
\[
(K(z,\zeta))_{z,\zeta \in F}
\]
is positive semi-definite. (In this article, $\Omega$ will be either
$\D^2$ or $\D$.)  The Agler decomposition generalizes the Pick
interpolation theorem from one-variable complex analysis, which
implies that for any $f: \D \to \D$, holomorphic,
\[
\frac{1-f(z)\overline{f(\zeta)}}{1-z\bar{\zeta}} 
\]
is a positive semi-definite kernel.

In recent years, more refined versions of the Agler decomposition have
been found for rational inner functions.  See \cite{CW99},
\cite{GW04}, \cite{gK08a}, or \cite{gK10aa}. (Unrelated to rational
inner functions, in specific, but still relevant are \cite{BSV05} and
\cite{LMP09}).  It has not been clear which of the ``refined'' aspects
of these decompositions for rational inner functions would extend to
more general bounded analytic functions (and which would actually be
useful).  The following theorem represents an offering in this
direction.  Our hope is that others may find it useful without having
to learn any of the underlying theory required to prove it.

\begin{theorem} \label{mainthm} 
Let $f: \D^2 \to \D$ be holomorphic. Then, there exist positive
semidefinite kernels $K_1, K_2$, and
holomorphic kernels $L_1, L_2$ such that
\[
1 - f(z)\overline{f(\zeta)} = (1-z_1\bar{\zeta}_1) K_1(z,\zeta) + (1-z_2 \bar{\zeta}_2) K_2 (z,\zeta)
\]
and
\[
f(z) - f(\zeta) = (z_1-\zeta_1)L_1(z,\zeta) + (z_2-\zeta_2)
L_2 (z,\zeta),
\]
where $(z,\zeta)=((z_1,z_2),(\zeta_1,\zeta_2))$.  In addition, the
following (pointwise) inequalities hold
\[
|L_j(z, \zeta)|^2 \leq
K_j(z,z)K_j(\zeta,\zeta) 
\]
for $j=1,2$.
\end{theorem}
Notice $L_j(z,z) = \frac{\partial f}{\partial z_j}(z)$.  So, estimates
on the positive semi-definite kernels in this decomposition provide
estimates on the derivatives of $f$. 

The analogous inequalities in one variable are
\[
\left|\frac{f(z) - f(\zeta)}{z-\zeta}\right|^2 \leq
\left|\frac{1-f(z)\overline{f(\zeta)}}{1-z\bar{\zeta}}\right|^2 \leq
\frac{1-|f(z)|^2}{1-|z|^2} \frac{1-|f(\zeta)|^2}{1-|\zeta|^2}
\]
which are consequences of the Schwarz-Pick lemma. 

As an application of this theorem, we are able to reprove a useful
theorem of \cite{GHW08} related to norm preserving extensions of
holomorphic functions on the polydisk and holomorphic retracts of the
polydisk. When working in $\D^{n+1}$ we will typically denote points
by $(z,w)$ where $z \in \D^n$ and $w \in \D$.

\begin{theorem}[\cite{GHW08}] \label{guothm}
Let $V \subset \D^{n+1}$, and suppose $w|_V$ has a nontrivial norm 1
holomorphic extension to $\D^n$.  Then, $V$ is a subset of the graph
of a holomorphic function of $z$.
\end{theorem}

Here \emph{nontrivial norm 1 extension} refers to a function on
$\D^{n+1}$ other than $w$ which agrees with $w$ on $V$ and whose
modulus has supremum norm at most 1.  Guo et al.'s proof involved an
interesting use of the one-variable Denjoy-Wolff theorem.  Guo et al.
used this result to continue some of the work initiated in the paper
\cite{AM03}.  Additionally, they reproved Heath and Suffridge's
characterization of holomorphic retracts of the polydisk, which we now
define.

\begin{definition}
A subset $V \subset \D^n$ is a \emph{holomorphic retract} if there
exists a holomorphic function (\emph{a retraction}) $\rho: \D^n \to
\D^n$ such that
\[
\rho \circ \rho = \rho \text{ and } \rho(\D^n) = V
\]
\end{definition}

Heath and Suffridge characterized all holomorphic retracts of the
polydisk as graphs.

\begin{theorem}[\cite{HS81}] \label{retractthm}
Suppose $V \subset \D^n$ is a holomorphic retract.  Then, after
applying an automorphism of $\D^n$, $V$ can be put into the form
\[
\{ (z, f(z)): z \in \D^k \} 
\]
where $f: \D^k \to \D^{n-k}$ is holomorphic.
\end{theorem}

The proof of Heath and Suffridge involves an impressive and technical
study of properties of Taylor series of retracts.  Guo et al. gave a
new proof by rehashing some of their proof of Theorem \ref{guothm}.
Although it is something of an aside, we think it is worth it to show
a slightly different approach in Section \ref{retractsec}.  While our
proofs are different from both Heath and Suffridge and Guo et al., the
general roadmap of our approach owes a great deal to Guo et al.

\section{Proof of Theorem \ref{mainthm}}

Let us first explain the result for rational inner functions and then
use an approximation argument to prove it for all analytic functions
bounded by one on $\D^2$.

As shown in \cite{wR69} (Theorem 5.2.5), every rational inner function
on $\D^2$ can be represented as
\[
f(z) = \frac{\tilde{p}(z_1,z_2)}{p(z_1,z_2)}
\]
where $p \in \C[z_1,z_2]$ has no zeros in $\D^2$, $\tilde{p}(z_1,z_2)
= z_1^{n} z_2^{m} \overline{p(1/\bar{z}_1, 1/\bar{z}_2)}$ for
appropriate powers $n,m$, and $\tilde{p}$ and $p$ have no common
factor. Necessarily $p$ and $\tilde{p}$ have bidegree at most $(n,m)$
(i.e. degree at most $n$ in $z_1$ and $m$ in $z_2$).

\cite{GW04} proved a detailed version of a two-variable
Christoffel-Darboux formula (see their Proposition 2.3.3 and also
\cite{CW99} and \cite{gK08a}), which can be stated as follows: there
exist polynomials $A_1,\dots, A_n \in \C[z_1,z_2]$ of bidegree at most
$(n-1,m)$ and polynomials $B_1,\dots, B_m \in \C[z_1,z_2]$ of bidegree
at most $(n,m-1)$ such that
\begin{equation} \label{SOS1}
p(z) \overline{p(\zeta)} - \tilde{p}(z) \overline{\tilde{p}(\zeta)} =
(1-z_1\bar{\zeta}_1) \sum_{j=1}^{n} A_j(z) \overline{A_j(\zeta)} +
(1-z_2\bar{\zeta}_2) \sum_{j=1}^{m} B_j(z) \overline{B_j(\zeta)}
\end{equation}
Let $\tilde{A}_j(z) := z_1^{n-1} z_2^m \overline{A_j(1/\bar{z}_1,
  1/\bar{z}_2)}$, $\tilde{B}_j(z) := z_1^{n} z_2^{m-1}
\overline{B_j(1/\bar{z}_1, 1/\bar{z}_2)}$.  If we perform a similar
reflection operation to \eqref{SOS1} (i.e. replace $(z,\zeta)$ with
$((1/\bar{z}_1, 1/\bar{z}_2), (1/\bar{\zeta}_1, 1/\bar{\zeta}_2))$,
take complex conjugates and multiply through by $z_1^n z_2^m
\bar{\zeta}_1^{n} \bar{\zeta}_2^{m}$) we get
\begin{equation} \label{SOS2}
p(z) \overline{p(\zeta)} - \tilde{p}(z) \overline{\tilde{p}(\zeta)} =
(1-z_1\bar{\zeta}_1) \sum_{j=1}^{n} \tilde{A}_j(z) \overline{\tilde{A}_j(\zeta)} +
(1-z_2\bar{\zeta}_2) \sum_{j=1}^{m} \tilde{B}_j(z) \overline{\tilde{B}_j(\zeta)}
\end{equation}
If we average \eqref{SOS1} and \eqref{SOS2} and rewrite using vector
notation, we get
\begin{equation} \label{SOS3}
p(z) \overline{p(\zeta)} - \tilde{p}(z) \overline{\tilde{p}(\zeta)} =
(1-z_1\bar{\zeta}_1) \ip{A(z)}{A(\zeta)} + (1-z_2\bar{\zeta}_2)
\ip{B(z)}{B(\zeta)} 
\end{equation}
where 
\[
A = \frac{1}{\sqrt{2}} [A_1,\dots, A_n, \tilde{A}_1,\dots,
  \tilde{A}_n]^t
\]
\[
B = \frac{1}{\sqrt{2}} [B_1,\dots, B_m, \tilde{B}_1,\dots,
  \tilde{B}_m]^t
\]
and $\ip{v}{w}=w^*v$ denotes the standard complex euclidean inner
product (with dimension taken from context).

If we reflect \eqref{SOS3} in $z$ alone we get
\begin{equation} \label{SOS4}
\tilde{p}(z) p(\zeta) - p(z) \tilde{p}(\zeta) = (z_1-\zeta_1)
\tilde{A}(z) \cdot A(\zeta) + (z_2-\zeta_2) \tilde{B}(z) \cdot
B(\zeta)
\end{equation}
where ``$\cdot$'' denotes the dot product: $v \cdot w = w^t v$.

Now, if we divide \eqref{SOS3} by $p(z)\overline{p(\zeta)}$ and divide
\eqref{SOS4} by $p(z) p(\zeta)$, we get equations of the form
\[
1- f(z) \overline{f(\zeta)} = \sum_{j=1}^2 (1-z_j\bar{\zeta}_j)
K_j(z,\zeta)
\]
\[
f(z) - f(\zeta) = \sum_{j=1}^2 (z_j - \zeta_j) L_j(z,\zeta)
\]
where $K_1,K_2$ are positive semidefinite kernels given explicitly by
\[
K_1(z,\zeta) = \frac{\ip{A(z)}{A(\zeta)}}{p(z) \overline{p(\zeta)}}
\qquad K_2(z,\zeta) = \frac{\ip{B(z)}{B(\zeta)}}{p(z)
  \overline{p(\zeta)}}
\]
and $L_1, L_2$ are holomorphic kernels given explicitly by
\[
L_1(z,\zeta) = \frac{\tilde{A}(z)\cdot A(\zeta)}{p(z) p(\zeta)} \qquad
L_2(z,\zeta) = \frac{\tilde{B}(z)\cdot B(\zeta)}{p(z) p(\zeta)}.
\]
The inequality
\[
|L_j(z,\zeta)|^2 \leq K_j(z,z)K_j(\zeta,\zeta)
\]
follows from Cauchy-Schwarz and the fact that $|A| = |\tilde{A}|$ and
$|B| = |\tilde{B}|$.  

This proves the theorem for rational inner functions.

It is proven in \cite{wR69} (Theorem 5.5.1) that holomorphic functions
$f: \D^2 \to \D$ can be approximated locally uniformly by rational
inner functions.  So, let $\{f^{(i)}\}_i$ be a sequence of rational
inner functions converging locally uniformly to $f$ with corresponding
$K^{(i)}_1, K^{(i)}_2, L^{(i)}_i, L^{(i)}_2$ satisfying the above
formulas/inequalities.  Because of the inequalities
\[
\begin{aligned}
|L^{(i)}_j(z,\zeta)|^2, |K^{(i)}_j(z,\zeta)|^2 &\leq
K^{(i)}_j(z,z)K^{(i)}_j(\zeta,\zeta) \\ &\leq
\frac{1}{(1-|z_1|^2)(1-|z_2|^2)(1-|\zeta_1|^2)(1-|\zeta_2|^2)}
\end{aligned}
\]
the kernel functions are locally bounded and hence form a normal
family.  We can select subsequences so that $K^{(i)}_1 \to K_1$,
$K^{(i)}_2 \to K_2$, $L^{(i)}_1 \to L_1$, $L^{(i)}_2 \to L_2$ locally
uniformly.  Positive semi-definiteness and pointwise inequalities are
preserved under this limit and therefore the statement of the theorem
holds.

\section{Guo et al.'s extension theorem}

As an application we prove Theorem \ref{guothm} in the following
slightly more detailed form.  Except for uniqueness, this is contained
in \cite{GHW08}.  

\begin{theorem} \label{maincor}
Let $V \subset \D^{n+1}$ be a set with more than one $w$-value and let
$\pi V$ be the projection of $V$ onto the first $n$ coordinates.
Suppose $w|_V$ has a nontrivial norm 1 holomorphic extension $F$.
Then, there is a unique holomorphic $f:\D^n \to \D$ such that
$F(z,f(z)) = f(z)$ and $V = \{(z,f(z)): z \in \pi V\}$.
\end{theorem}

In this section we generally follow the convention of denoting points
in $\D^{n+1}$ by $(z,w)$ with $z \in \D^n$ and $w \in \D$.

\begin{lemma} If $f: \D^{n+1} \to \D$ is holomorphic, $\phi$ is an
  automorphism of $\D$, and there exists a $z_0 \in \D^n$ such that
  $f(z_0, w) = \phi(w)$ for all $w$, then $f(z,w) = \phi(w)$ for all
  $(z,w)$.
\end{lemma}

\begin{proof} We may assume $\phi(w) = w$ and $z_0 = (0,\dots,0)$.  Then,
\[
G(z,w) = \frac{f(z,w) - w}{1-\bar{w} f(z,w)}
\]
is holomorphic in $z$, $G(0,w) = 0$, and $|G| \leq 1$.

Write $|z|_{\infty}$ for the maximum modulus of the components of $z$.
By the Schwarz lemma,
\[
|G(z,w)|^2 \leq |z|_{\infty}^2
\]
and
\[
1-|z|_{\infty}^2 \leq 1-|G(z,w)|^2  =
\frac{(1-|w|^2)(1-|f(z,w)|^2)}{|1-\bar{w}f(z,w)|^2} \leq
\frac{1-|w|^2}{|1-\bar{w} f(z,w)|^2}
\]
and so
\[
|w - f(z,w)|^2 \leq |1-\bar{w}f(z,w)|^2 \leq \frac{1-|w|^2}{1-|z|_{\infty}^2}.
\]
Then, by the maximum principle
\[
\sup_{w \in r\D} |w-f(z,w)|^2 \leq \frac{1-r^2}{1-|z|_{\infty}^2}
\]
which implies $f(z,w) \equiv w$ after letting $r \nearrow 1$.
\end{proof}

\begin{lemma} \label{dichotomy}
Let $F: \D^{n+1} \to \D$ be holomorphic and suppose $F(z_0,w_0) = w_0$
at some point.  Necessarily, 
\begin{equation} \label{dicteq}
|\frac{\partial F}{\partial w}(z_0,w_0)| \leq 1.
\end{equation}
If equality holds in \eqref{dicteq}, then $F(z, w) = \phi(w)$ for
some automorphism $\phi$.  If strict inequality holds in
\eqref{dicteq}, then there exists a unique holomorphic function
$f:\Omega \to \D$ defined in a neighborhood $\Omega$ of $z_0$ such
that $f(z_0) = w_0$ and $F(z,f(z)) = f(z)$ where defined.
\end{lemma}

\begin{proof} By the Schwarz lemma
\[
\frac{1 - |F(z_0,w_0)|^2}{1-|w_0|^2} = 1 \geq |\frac{\partial
  F}{\partial w} (z_0,w_0)|.
\]
If equality occurs then $F(z_0,w)$ is an automorphism of $\D$ and by
the previous lemma $F(z,w) = \phi(w)$ identically.  

If equality does not hold, then setting $G(z,w) = F(z,w) - w$ we see
that $\frac{\partial G}{\partial w}(z_0,w_0)  = \frac{\partial
  F}{\partial w}(z_0,w_0) -1 \ne 0$.  By the implicit function
theorem, there exists a function of $z$ in a neighborhood of $z_0$
such that $G(z,f(z)) = 0$; i.e. $F(z,f(z)) = f(z)$. 

To see that $f$ is unique, we note that if $F(z_1,w_1) = w_1$, there
cannot be a different $w_2 \ne w_1$ such that $F(z_1, w_2) = w_2$, for
then $F(z_1, w) \equiv w$ and hence $F(z,w) \equiv w$.  By assumption
this cannot occur, so any point $(z_1,w_1)$ satisfying $F(z_1,w_1) =
w_1$ is uniquely determined by the $z$ component.  In particular,
$F(z,f(z)) = f(z)$ cannot hold for two different choices of $f:\Omega
\to \D$.
\end{proof}

The final lemma is the most important and it utilizes the main
theorem.

\begin{lemma} \label{cruciallem}
Let $F: \D^{n+1} \to \D$ be holomorphic.  Suppose $F(z_0,w_0) = w_0$,
$F(z_1,w_1) = w_1$, $F(z_2,w_2) \ne w_2$, where $w_0 \ne w_1$.  Then
there exists a unique $f:\D^n \to \D$ such that $F(z,f(z)) = f(z)$.
In particular, if $F(z_3,w_3) = w_3$, then $f(z_3) = w_3$.
\end{lemma}

\begin{proof}
By the three assumptions, $F$ cannot be an automorphism as a function
of $w$.  So, we are in the second case of the previous lemma and there
locally (say on a domain $\Omega \subset \D^n$) exists a unique
$f:\Omega \to \D$ satisfying $F(z,f(z)) = f(z)$.  We need to extend
$f$ to all of $\D^n$.

We will show $f$ can be extended one variable at a time.  Letting
$\zeta =(\zeta_1,\zeta_2,\dots,\zeta_n)=(\zeta_1,\zeta') \in \Omega$,
we plan to show $f$ can be extended to $\D\times \{\zeta'\}$ in such a
way that the identity $F(z,f(z)) = f(z)$ is preserved.  By Lemma
\ref{dichotomy}, the identity will then extend to a unique function on
an open neighborhood of $\D\times \{\zeta'\}$.  So, given any other
point $\eta = (\eta_1,\dots, \eta_n)$, we will be able to successively
extend $f$ to $(\eta_1,\zeta_2,\dots,\zeta_n),
(\eta_1,\eta_2,\zeta_3,\dots), \dots, (\eta_1,\dots,\eta_n)$.  This
will imply $f$ is holomorphic on all of $\D^n$ and $F(z,f(z))=f(z)$.

For this argument we will use $t$ for the first coordinate of $z$ and
write $z = (t,z')$ (we are avoiding ``$z_1$'' since we have used this in a
different way in the lemma statement).  

Let $g(t) = f(t,\zeta')$ and $G(t,w) = F(t, \zeta', w)$.  Now $g$ is
holomorphic in some neighborhood of $\zeta_1$ and $G(t,g(t)) = g(t)$
holds in said neighborhood.  If $g$ is constant, then clearly $g$
extends to be holomorphic on $\D$ and $G(t,g(t)) = g(t)$ holds on all
of $\D$.  So, suppose $g$ is nonconstant.  Perturb $\zeta_1$ if
necessary to make $g'(\zeta_1) \ne 0$, and let $\partial_t,
\partial_w$ denote the partial derivatives with respect to $t,w$,
respectively.

Then, $\partial_t G(t,g(t))+\partial_w G(t,g(t)) \partial_t g(t) =
\partial_t g(t)$, so $\partial_t G(t,g(t)) = \partial_t
g(t)(1-\partial_w G(t,g(t)))$. Now, $\partial_t G(\zeta_1,g(\zeta_1))
\ne 0$ by the previous lemma (i.e. $\partial_w G(t,g(t))$ cannot equal
$1$, since this would imply $G$ is an automorphism as a function of
$w$) and since $\partial_t g(\zeta_1) \ne 0$.

We apply the main theorem to $G(t,w)$.  Theorem \ref{mainthm} implies
\[
1-G(t,w) \overline{G(\tau,\eta)} = (1-t\bar{\tau})K_1 +
(1-w\bar{\eta})K_2
\]
with $K_1,K_2$ positive semi-definite, where $K_1,K_2$ should be
evaluated at $((t,w),(\tau,\eta))$.  

Substituting $w = g(t), \eta = g(\tau)$ for $t,\tau$ in a neighborhood
of $\zeta_1$, and writing $v(t,\tau) = ((t,g(t)),(\tau,g(\tau)))$ for
short, we get
\[
1-g(t)\overline{g(\tau)} = (1-t\bar{\tau})K_1(v(t,\tau)) +
(1-g(t)\overline{g(\tau)})K_2(v(t,\tau)) 
\]
or
\begin{equation} \label{oreq}
(1-g(t)\overline{g(\tau)})(1-K_2(v(t,\tau))) =
(1-t\bar{\tau})K_1(v(t,\tau))
\end{equation}

We cannot have $K_2(v(\zeta_1,\zeta_1)) = 1$ for then
$K_1(v(\zeta_1,\zeta_1)) = 0$, which by the main theorem implies a
contradiction.  Specifically,
\[
|\partial_t G(\zeta_1,g(\zeta_1))| = |L_1(v(\zeta_1,\zeta_1))| \leq
|K_1(v(\zeta_1,\zeta_1))| = 0
\]
which is not the case as $\partial_t G(\zeta_1,g(\zeta_1)) \ne 0$.

So, $|K_2(v(t,\tau))| < 1$ for $t,\tau$ in some open set around
$\zeta_1$.  By \eqref{oreq}, for such $t, \tau$
\[
\frac{1-g(t)\overline{g(\tau)}}{1-t\bar{\tau}} = \frac{K_1(v(t,\tau))}
     {1-K_2(v(t,\tau))} = K_1(v(t,\tau)) \sum_{j=0}^{\infty}
     K_2(v(t,\tau))^j
\]
is positive semi-definite.  By the Pick interpolation theorem, $g$
extends to be holomorphic on all of $\D$. (See \cite{AM02} for the
Pick interpolation theorem from this point of view.)  Also, $G(t,g(t))
= g(t)$ then automatically holds on all of $\D$ by analyticity.  This
completes the proof.
\end{proof}

Theorem \ref{maincor} is just a rephrasal of this lemma.

\section{Retracts and Theorem \ref{retractthm}} \label{retractsec}
We prove the following refinement of Theorem \ref{retractthm} (which
can also be found in \cite{GHW08}).  

\begin{theorem} \label{refineretract}
Suppose $V \subset \D^n$ is a holomorphic retract.  Then, after
applying an automorphism of $\D^n$, $V$ can be put into the form
\[
\{ (z,e(z), f(z)): z \in \D^k \} 
\]
where $e: \D^k \to \D^m$ is a coordinate function in each component
and $f: \D^k \to \D^{n-m-k}$ is holomorphic with no components equal
to an automorphism as a function of a single variable.
\end{theorem}

An example for $e$ might be $e(z_1,z_2) = (z_1,z_1,z_1,z_2,z_2)$.  

\begin{lemma} \label{keylemma}
Suppose $V \subset \D^{n+1}$ is a holomorphic retract, with retraction
$\rho(z,w) = (\rho_1, \dots, \rho_n, \rho_{n+1}) = (\rho',
\rho_{n+1})$ where we assume $\rho_{n+1}$ is not an automorphism as a
function of one variable.  Then, there exists $f: \D^n \to \D$,
holomorphic, such that $(z,w) \mapsto (\rho'(z, f(z)), f(z))$ is a
retraction of $V$,
\begin{equation} \label{lemmaeq}
V = \{ (z,f(z)): z \in \pi V\}
\end{equation}
and $\pi V$ is a retract with retraction $z \mapsto \rho'(z,f(z))$.

\end{lemma}

\begin{proof}
If $\rho_{n+1}$ is constant, there is nothing to prove, so assume
otherwise.  Then, $\rho_{n+1}(z,w) = w$ for two distinct values of $w$
(and necessarily different values of $z$) since $\rho_{n+1}(\rho',
\rho_{n+1}) = \rho_{n+1}$.  Therefore, Theorem \ref{maincor} applies.
There exists $f:\D^n \to \D$ holomorphic satisfying
\begin{equation} \label{frho}
\rho_{n+1}(z,f(z)) = f(z)
\end{equation}
and 
\begin{equation} \label{frho2}
\{(z,w): \rho_{n+1}(z,w) = w \} = \{(z,f(z)): z \in \D^n\} \supset V.
\end{equation}
This proves \eqref{lemmaeq}.  

As $\rho(z,f(z)) \in V$ we see that by \eqref{frho2} and \eqref{frho},
$f(\rho'(z,f(z))) = \rho_{n+1}(z,f(z)) = f(z)$, which shows $(z,w)
\mapsto (\rho'(z,f(z)), f(z))$ agrees with the map $(z,w) \mapsto
\rho(z,f(z))$.  This is a retraction since its composition with itself
is
\[
\begin{aligned}
\rho(\rho'(z,f(z)), f(\rho'(z,f(z)))) & =
\rho(\rho'(z,f(z)),\rho_{n+1}(z,f(z))) \\
&= \rho(\rho(z,f(z))=
\rho(z,f(z))
\end{aligned}
\]
as desired.

We need to show that the range of $(z,w)\mapsto \rho(z,f(z))$ contains
$V$ (it certainly is contained in $V$).  If $(z,w) \in V$, then
$w=f(z)$ and $\rho(z,f(z)) = (z,f(z))= (z,w)$. So, this map is a
retraction \emph{of $V$}.

Finally, we must show $\pi V$ is a retract with retraction $z \mapsto
\rho'(z,f(z))$. This map is indeed a retraction since $(z,w) \mapsto
(\rho'(z,f(z)),f(z))$ is, and the first $n$ components necessarily
trace out $\pi V$.
\end{proof}

\begin{lemma} \label{autolem}
Let $\rho = (\rho_1,\dots, \rho_n): \D^n \to \D^n$ be a retraction of
$V$.  If $\rho_1(z_1,\dots, z_n)$ is an automorphism as a function of
$z_1$, then $\rho_1(z) \equiv z_1$.  If $\rho_1$ is an automorphism as
a function of $z_2$ then $\rho_2(z) \equiv z_2$ and $\rho_1(z) \equiv
\phi(z_2)$ for some $\phi$.
\end{lemma}

\begin{proof}
If $\rho_1$ is an automorphism, say $\phi$, as a function
of $z_1$, then $\phi \circ \phi = \phi$, which means $\phi =
\text{id}$.  This means $\rho_1(z) = z_1$. If $\rho_1$ is an
automorphism, say $\phi$, as a function of $z_2$, then
$\phi(\rho_2(z)) = \rho_1(\rho(z)) = \rho_1(z) = \phi(z_2)$.  This
implies $\rho_2(z) = z_2$.
\end{proof}

\begin{lemma} \label{formlem}
Let $\rho=(\rho_1,\dots, \rho_n):\D^n \to \D^n$ be a retraction of $V$
with all components equal to an automorphism as a function of a single
variable.  After conjugating by automorphisms of $\D^n$ we may put
$\rho$ into the form
\[
\rho(z,w) = (z, e(z))
\]
where $z \in \D^k$, $w \in \D^{n-k}$, $e: \D^k \to \D^{n-k}$, where
each component of $e$ is a coordinate function.
\end{lemma}

\begin{proof} 
Let us just illustrate.  If $\rho_2(z) = z_2$ and $\rho_1(z) =
\phi(z_2)$ for some one variable automorphism $\phi$, we can conjugate
by the automorphism of $\D^n$ given by $\psi(z) = (\phi^{-1}(z_1),
z_2, \dots, z_n)$ to get
\[
\psi \circ \rho \circ \psi^{-1}(z) = (z_2,z_2,\rho_3\circ \psi^{-1}(z),\dots, \rho_n\circ \psi^{-1}(z))
\]
The lemma then follows from the previous lemma after reordering and
conjugating by analogous automorphisms as necessary.
\end{proof}

\begin{proof}[Proof of Theorem \ref{refineretract}]
There is no harm in assuming $V$ is not a Cartesian product of a point
and a retract (this is equivalent to assuming our retractions do not
possess a constant component).  

We proceed by induction. Let $n=1$ and let $\rho: \D \to V$ be a
retraction.  One variable retractions are either constant (which by
assumption is ruled out) or equal to the identity (by the Schwarz-Pick
lemma, a self-map of the disk with two fixed points equals the
identity).  

Suppose the theorem holds for $n$ dimensional retracts.  Let $\rho:
\D^{n+1} \to V$ be a retraction onto $V$.  
If all components of $\rho$ are automorphisms (in a single variable)
then we are finished by Lemma \ref{formlem}.
So, we assume some component is not an automorphism in a single
variable and relabel to make $\rho_{n+1}$ such a component.  By Lemma
\ref{keylemma}, we can replace $\rho$ with a retraction $r$ of the
form
\[
r(z,w) = (r'(z), f(z))
\]
where $z \in \D^n$ and $w \in \D$ and the projection of $V$ onto the
first $n$ coordinates, denoted $\pi V$, is a retract with retraction
$r'(z)$.  By induction (after possibly applying automorphisms of
$\D^n$) we may put $\pi V$ into the form
\[
\pi V = \{ (z,e(z), g(z)): z \in \D^k\}
\]
where $e: \D^k \to \D^{m}$ consists of coordinate functions, $g: \D^k
\to \D^{n-m-k}$ is holomorphic with no components equal to an
automorphism.  

Then, 
\[
V = \{ (z,e(z), g(z), f(z,e(z),g(z))): z \in \D^k\}
\]
which is of the desired form.
\end{proof}

\bibliography{aradaga}

\end{document}